\documentclass[12pt,leqno]{amsart}
\usepackage{amssymb}
\usepackage{euscript}
\usepackage[hypertex]{hyperref}
\usepackage{colordvi}

\title[$n$-ary algebras -- calculations and
      conjectures]{Operads for $\hbox{\Large
      $n$}$-ary algebras -- calculations and conjectures}

\author[Markl - Remm]{Martin Markl and Elisabeth Remm}
\thanks{The first author was supported by the grant GA \v CR
                  201/08/0397 and by
   the Academy of Sciences of the Czech Republic,
   Institutional Research Plan No.~AV0Z10190503.}
\address{Mathematical Institute of the Academy, {\v Z}itn{\'a} 25,
         115 67 Prague 1, The Czech Republic}
\email{markl@math.cas.cz}
\subjclass[2000]{18D50, 55P48}

\keywords{Operad, Koszulity, minimal model}

\address{Laboratoire de Math\'ematiques et Applications,
        Universit\'e de Haute Alsace, Facult\'e des Sciences et
        Techniques, 4, rue des Fr\`eres Lumi\`ere,
        68093~Mulhouse~cedex, France.}
\email{Elisabeth.Remm@uha.fr}

\begin{document}
\baselineskip18pt plus 1pt minus 1pt
\parskip3pt plus 1pt minus .5pt

\def\u{\underline}
\def\miska{
\unitlength 1.2mm 
\begin{picture}(4.2,9)(0,0)
\thicklines
\put(-8,-8){
\put(8,9){\line(0,-1){1}}
\put(8,8){\line(1,0){4}}
\put(12,8){\line(0,1){1}}
}
\end{picture}
}
\def\miskaot{
\unitlength 1.2mm 
\begin{picture}(4.2,9)(0,0)
\thicklines
\put(-8,-8){\put(9,9){?}
\put(8,9){\line(0,-1){1}}
\put(8,8){\line(1,0){4}}
\put(12,8){\line(0,1){1}}
}
\end{picture}
}
\def\p[#1]{\langle #1 \rangle}\def\t#1{\hskip .2em t^{#1}}
\def\exepttree{\hskip .5em\raisebox{-.2em}{\rule{.8pt}{1.1em}}  \hskip .5em}
\def\card{{\rm card}}\def\calC{{\mathcal{C}}}\def\rada#1#2{#1,\ldots,#2}
\def\Rada#1#2#3{#1_{#2},\dots,#1_{#3}}\def\Dbar{{\sf D}}
\def\epi{ \twoheadrightarrow}\def\Sree#1{{\EuScript{S}^n_{#1}}}
\def\Sreep#1{{\EuScript{S}^3_{#1}}}
\def\Free{{\pAss^n_d(V)}}\def\Tree#1{{\EuScript{T}^n_{#1}}}
\def\q{{\hbox{$\bullet$}}}\def\qq{{\hbox{\small$\bullet$}}}
\def\dst{\delta_{\rm st}}\def\Cst#1{{C_{\sstildeAss}^{#1}(A;A)_{\rm st}}}
\def\PCst#1{{C_{\calP}^{#1}(A;A)_{\rm st}}}\def\sspAss{{p\ssAss}}
\def\PC#1{{C_{\calP}^{#1}(A;A)}}\def\sstAss{{t\ssAss}}
\def\C#1{{C^{#1}_{\sstildeAss}(A;A)}}
\def\scrR{{\mathscr R}}\def\pa{{\partial}}\def\Hom{{\it Hom\/}}
\def\H#1{{H_{\sstildeAss}^{#1}(A;A)}}
\def\Hst#1{{H_{\sstildeAss}^{#1}(A;A)_{\rm st}}}
\def\PHst#1{{H_{\calP}^{#1}(A;A)_{\rm st}}}
\def\PH#1{{H_{\calP}^{#1}(A;A)}}
\def\bbbC{{\mathbb C}}\def\frakZ{{\mathfrak Z}}\def\bbbR{{\mathbb R}}
\def\martin{\noindent{\bf Martin:\ }}
\def\endmartin{ \hfill\rule{10mm}{.75mm} \break}
\def\otexp#1#2{{#1}^{\otimes #2}}\def\ot{\otimes}
\def\ssAss{\hbox{\normalsize $\mathcal{A}\it ss$}}
\def\Ass{\hbox{$\mathcal{A}\it ss$}}
\def\bfk{{\mathbf k}}\def\tAss{{t\Ass}}
\def\tildeAss{\widetilde{\Ass}}
\def\sstildeAss{\widetilde{\ssAss}}\def\ssttildeAss{t\widetilde{\ssAss}{}}
\def\id{\hbox{$1 \hskip -.3em 1$}}\def\ssptildeAss{p\widetilde{\ssAss}{}}
\def\calP{{\mathcal P}}\def\Span{{\it Span}}\def\pAss{{p\Ass}}
\def\ss{{\mathbf s}}\def\Associative{\Ass}\def\osusp{\hbox{\bf s\hskip .1em}}
\def\ttildeAss{{t\tildeAss}{}}\def\ptildeAss{{p\tildeAss}{}}
\def\ssosusp{\hbox{\scriptsize\bf s\hskip .1em}}
\def\gl#1{
{
\unitlength=.5pt
\begin{picture}(60.00,0.00)(0.00,0.00)
\thicklines
\put(2.00,-32.00){\makebox(0.00,0.00)[b]{\scriptsize $#1$}}
\put(0.00,0.00){\line(-1,-2){17.00}}
\put(0.00,0.00){\line(1,-2){17.00}}
\put(-18,-35.00){\line(1,0){35.00}}
\put(-15,-35.00){\line(0,-1){8.00}}
\put(-11,-35.00){\line(0,-1){8.00}}
\put(15,-35.00){\line(0,-1){8.00}}
\put(3,-37.00){\makebox(0.00,0.00)[t]{\scriptsize $\cdots$}}
\end{picture}}
}

\def\cases#1#2#3#4{
                  \left\{
                         \begin{array}{ll}
                           #1,\ &\mbox{#2}
                           \\
                           #3,\ &\mbox{#4}
                          \end{array}
                   \right.
}

\def\muu{
\hskip .2em
\thicklines
{
\unitlength=.2pt
\begin{picture}(60.00,20.00)(0.00,10.00)
\put(30.00,30.00){\makebox(0.00,0.00){$\bullet$}}
\put(30.00,30.00){\line(1,-1){30.00}}
\put(0.00,0.00){\line(1,1){30.00}}
\put(30.00,30.00){\line(0,1){35}}
\end{picture}}
\hskip .2em
}

\def\ZZbbZbZ{
\thicklines
{
\unitlength=.3pt
\begin{picture}(80.00,40.00)(0.00,40.00)
\put(40.00,40.00){\makebox(0.00,0.00){$\bullet$}}
\put(20.00,20.00){\makebox(0.00,0.00){$\bullet$}}
\put(40.00,0.00){\line(-1,1){20.00}}
\put(40.00,40.00){\line(0,1){40.00}}
\put(80.00,0.00){\line(-1,1){40.00}}
\put(0.00,0.00){\line(1,1){40.00}}
\end{picture}}
\raisebox{-02pt}{\rule{0pt}{2pt}}
}

\def\ZbZbbZZ{
\thicklines
{
\unitlength=.3pt
\begin{picture}(80.00,40.00)(0.00,40.00)
\put(40.00,40.00){\makebox(0.00,0.00){$\bullet$}}
\put(60.00,20.00){\makebox(0.00,0.00){$\bullet$}}
\put(40.00,0.00){\line(1,1){20.00}}
\put(40.00,40.00){\line(0,1){40.00}}
\put(80.00,0.00){\line(-1,1){40.00}}
\put(0.00,0.00){\line(1,1){40.00}}
\end{picture}}
\raisebox{-20pt}{\rule{0pt}{2pt}}
}

\swapnumbers
\newtheorem{theorem}{Theorem}[section]
\newtheorem{corollary}[theorem]{Corollary}
\newtheorem{observation}[theorem]{Observation}
\newtheorem{lemma}[theorem]{Lemma}
\newtheorem{proposition}[theorem]{Proposition}
\newtheorem{problem}[theorem]{Problem}
\newtheorem{conjecture}[theorem]{Conjecture}
\newtheorem{odstavec}{\hskip -0.1mm}[section]
\newtheorem*{principle}{Principle}
\newtheorem*{itt}{Invariant Tensor Theorem}

\theoremstyle{definition}
\newtheorem{example}[theorem]{Example}
\newtheorem{remark}[theorem]{Remark}
\newtheorem{definition}[theorem]{Definition}

\newtheorem*{conjectureA}{Conjecture A}
\newtheorem*{conjectureAp}{Conjecture A'}
\newtheorem*{conjectureApp}{Conjecture A''}
\newtheorem*{conjectureB}{Conjecture B}
\newtheorem*{sublemma}{Sublemma}
\newtheorem*{fact}{Fact}

\def\u{\relax}

\pagestyle{myheadings}

\begin{abstract}
In~\cite{markl-remm} we studied Koszulity of a family $\tAss^n_d$ of
operads depending on a natural number $n \in \mathbb N$ and on the
degree $d \in \mathbb Z$ of the generating operation.  While we proved
that, for $n \le 7$, the operad $\tAss^n_d$ is Koszul if and only if
$d$ is even, and while it follows from \cite{hoffbeck} that
$\tAss^n_d$ is Koszul for $d$ even and arbitrary $n$, the
(non)Koszulity of $\tAss^n_d$ for $d$ odd and $n \geq 8$ remains an
open problem.  In this note we describe some related numerical
experiments, and formulate a conjecture suggested by the results of
these computations.
\end{abstract}

\maketitle

\bibliographystyle{plain}

\def\job{\vfill
         \hfill{\tt \jobname.tex}}
\def\zvetseni{\Large}

\section{Introduction}

All algebraic objects will be considered over a ground field ${\mathbf
k}$ of characteristic zero. In particular, the symbol $\otimes$ will
denote the tensor product over ${\mathbf k}$.  We assume some
familiarity with operad theory, namely with Koszul duality for
quadratic operads and their Koszulity, see for 
instance~\cite[Chapter~II.3]{markl-shnider-stasheff:book} or the original
sources \cite{getzler-jones:preprint,ginzburg-kapranov:DMJ94}. In
Section~\ref{calc} we also refer to minimal models for operads. The
necessary notions can again be found
in~\cite[Chapter~II.3]{markl-shnider-stasheff:book} or in the original
source~\cite{markl:zebrulka}. We however recall the most basic notions
at the beginning of Section~\ref{sec:ginzb-kapr-crit-2}.

The operad  $\tAss^n_d$ mentioned in the abstract describes algebras
introduced in the following:

\begin{definition}
Let $V$ be a graded vector space, $n \geq 2$,
and $\mu : \otexp Vn \to V$ a degree $d$ linear map.
The couple $A = (V,\mu)$ is a {\em degree $d$ totally associative
$n$-ary algebra\/} if, for each $1 \leq i,j \leq n$,  
\begin{equation}
\label{Jarca}
\mu\left(\id^{\otimes i-1} \otimes \mu \otimes \id^{\otimes
  n-i}\right) 
= \mu\left(\id^{\otimes j-1} \otimes \mu \otimes \id^{\otimes
  n-j}\right),
\end{equation}
where $\id : V \to V$ denotes the identity map. 
\end{definition}
 
If we symbolize $\mu$ by an oriented corolla with one output and $n$
inputs, then the axiom~(\ref{Jarca}) can be depicted as
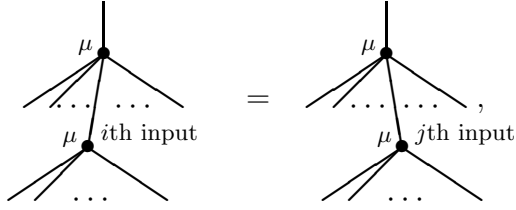
\begin{center}
\raisebox{-11pt}{\hphantom{.}}
\unitlength=1pt
\begin{picture}(60.00,70.00)(0.00,0.00)
\thicklines
\put(42.00,30.00){\makebox(0.00,0.00){$\cdots$}}
\put(30.00,50.00){\makebox(0.00,0.00){$\bullet$}}
\put(26.00,50.00){\makebox(0.00,0.00)[rb]{\scriptsize  $\mu$}}
\put(30.00,50.00){\line(3,-2){30.00}}
\put(30.00,50.00){\line(-1,-1){20.00}}
\put(30.00,50.00){\line(-3,-2){30.00}}
\put(30.00,50.00){\line(0,1){20.00}}
\put(19.5,30.00){\makebox(0.00,0.00){$\cdots$}}
\put(29.00,24.00){\makebox(0.00,0.00)[lt]{\scriptsize  $i$th input}}
\put(30.00,50.00){\line(-1,-6){5.70}}
\put(-6,0){
\put(32.00,-5.00){\makebox(0.00,0.00){$\cdots$}}
\put(30.00,15.00){\makebox(0.00,0.00){$\bullet$}}
\put(26.00,15.00){\makebox(0.00,0.00)[rb]{\scriptsize  $\mu$}}
\put(30.00,15.00){\line(3,-2){30.00}}
\put(30.00,15.00){\line(-1,-1){20.00}}
\put(30.00,15.00){\line(-3,-2){30.00}}
}
\end{picture}
\hskip 20pt \raisebox{30pt}{=} \hskip 10pt
\unitlength=1pt
\begin{picture}(60.00,70.00)(0.00,0.00)
\thicklines
\put(44.00,30.00){\makebox(0.00,0.00){$\cdots$}}
\put(30.00,50.00){\makebox(0.00,0.00){$\bullet$}}
\put(26.00,50.00){\makebox(0.00,0.00)[rb]{\scriptsize  $\mu$}}
\put(30.00,50.00){\line(3,-2){30.00}}
\put(30.00,50.00){\line(-1,-1){20.00}}
\put(30.00,50.00){\line(-3,-2){30.00}}
\put(30.00,50.00){\line(0,1){20.00}}
\put(24.00,30.00){\makebox(0.00,0.00){$\cdots$}}
\put(41.00,24.00){\makebox(0.00,0.00)[lt]{\scriptsize $j$th input}}
\put(30.00,50.00){\line(1,-6){5.70}}
\put(6,0){
\put(32.00,-5.00){\makebox(0.00,0.00){$\cdots$}}
\put(30.00,15.00){\makebox(0.00,0.00){$\bullet$}}
\put(26.00,15.00){\makebox(0.00,0.00)[rb]{\scriptsize  $\mu$}}
\put(30.00,15.00){\line(3,-2){30.00}}
\put(30.00,15.00){\line(-1,-1){20.00}}
\put(30.00,15.00){\line(-3,-2){30.00}}
}
\hskip 64pt \raisebox{30pt}{,}
\end{picture}
\end{center}
with the compositions of the indicated operations taken from the
bottom up.

Therefore, in totally associative algebras, all associations of the
iterated \hbox{$n$-ary} multiplication are the same.
Degree $0$ totally associative $2$-algebras are
ordinary associative algebras.  Degree $0$ totally associative $n$-algebras are
usually called simply $n$-ary totally associative algebras.

Let $\tAss^n_d$ be the {\em operad\/} for  degree $d$ totally associative 
$n$-algebras. It is not difficult to prove that 
the Koszulity of $\tAss^n_d$ depends only on the
{\em parity\/} of $d$. In this brief note we focus on

\begin{conjectureA}
The operad $\tAss^n_d$ is Koszul if and only if $d$ is even.
\end{conjectureA}

It follows from the work of Hoffbeck~\cite{hoffbeck} on the
{Poincar\'e -Birkhoff-Witt} 
criterion for operads that $\tAss^n_d$ \underline{is} Koszul for $d$
even. In~\cite{markl-remm} we proved that $\tAss^n_d$ \underline{is}
\underline{not} Koszul if $d$ is odd and $n \leq 7$. The non-Koszulity
for $d$ odd and $n \geq 8$ is therefore still \u{con}j\u{ectural}.

\section{Ginzburg-Kapranov's criterion for $n$-ary operads}
\label{sec:ginzb-kapr-crit-2}

For convenience of the reader we recall, following~\cite{markl-remm},
some features of the Koszul duality of {\em non-binary\/} operads.
Assume $E = \{E(a)\}_{a \geq 2}$ is a $\Sigma$-module of finite type
concentrated in arity~$n$.  Operads $\calP = \Gamma(E)/(R)$, where
$\Gamma(E)$ is the free operad on $E$ and $(R)$ the ideal generated by
a subspace $R \subset \Gamma(E)(2n-1)$ are called {\em $n$-\u{ar}y
q\u{uadratic}\/}.  Let $E^\vee = \{E^\vee(a)\}_{a
\geq 2}$ be the $\Sigma$-module with
\[
E^\vee(a) := \cases{\mbox{sgn}_a \otimes \uparrow^{a-2}E(a)^\#}
                   {if $a = n$ and}0{otherwise}
\]
where $\uparrow^{a-2}$ is the iterated suspension, 
$\mbox{sgn}_a$ the signum representation, 
and $\#$ the linear dual of a graded
vector space with the induced representation.
There is a non-degenerate pairing
\[
\langle - | - \rangle :
\Gamma (E^\vee)(2n-1) \otimes \Gamma (E)(2n-1) \to \bfk.
\]
Its concrete form is not relevant for this note, the details can be
found in \cite[page~142]{markl-shnider-stasheff:book}.

\begin{definition}
The {\em Koszul dual\/} of the $n$-ary operad
$\calP = \Gamma(E)/(R)$  is the quotient
\[
\calP^!  := \Gamma (E^\vee)/(R^\perp),
\] 
where $R^\perp \subset \Gamma
(E^\vee)(2n-1)$ is the annihilator of $R \subset
\Gamma(E)(2n-1)$ in the above pairing, and $(R^\perp)$ the
ideal generated by $R^\perp$.
\end{definition}

If $\calP$ is $n$-ary, generated by
an operation of degree $d$, then the generator
of $\calP^!$ has the same arity but
degree $-d + n-2$, i.e.~for $n\not =2$ (the non-binary case) the Koszul
duality \u{ma}y \u{not} preserve the degree of the generating
operation.
In  the following standard
definition, $\Dbar(-)$ denotes the dual bar construction
\cite{ginzburg-kapranov:DMJ94}.

\begin{definition}
\label{Zitra_budu_s_Jaruskou!}
A quadratic operad $\calP$ is {\em Koszul\/} if the natural map
$\Dbar(\calP^!) \to \calP$ is a homology equivalence.
\end{definition}

The definition below describes algebras over the Koszul
dual of $\tAss^n_d$.

\begin{definition}
\label{sec:ginzb-kapr-crit}
Let
$V$ be a graded vector space and $\mu : \otexp Vn \to V$ 
a degree $d$ linear map. The couple 
$A = (V,\mu)$ is  a {\em \u{de}g\u{ree} $d$ p\u{artiall}y \u{associative}
$n$-\u{ar}y \u{al}g\u{ebra}\/} if the following single axiom is satisfied:  
\[
\label{Jaruska_sibalsky_mrka.}
\sum_{i=1}^n (-1)^{(i+1)(n-1)}
\mu\left(\id^{\otimes i-1} \otimes \mu \otimes \id^{\otimes n-i}\right) =0.
\]
\end{definition}

In partially associative $n$-ary algebras, all
associations of the multiplication (with alternating signs if
$n$ is even) sum to zero. So, for $n = 2$ one has 
$
((ab)c) - (a(bc)) = 0,
$
thus degree $d$ partially associative $2$-ary algebras are
associative algebras with multiplication of degree $d$.
For $n=3$ one has
\[
((abc)de) + (a(bcd)e) +  (ab(cde))=0.
\]

Degree $(n-2)$ partially associative $n$-ary algebras are
$A_\infty$-algebras $A =
(V,\mu_1,\mu_2,\ldots)$~\cite[\S1.4]{markl:JPAA92} which are {\em
meager\/} in that they satisfy $\mu_k = 0$ for $k \not=n$.  Their
symmetrizations are {\em Lie\/} 
$n$-{\em algebras}~\cite{hanlon-wachs:AdvMa:95}.

Let $\pAss^n_d$ denote the operad for \u{de}g\u{ree} $d$ p\u{artiall}y
\u{associative} $n$-\u{ar}y \u{al}g\u{ebras}. The following statement
follows from a simple calculation.

\begin{proposition}
One has isomorphisms of operads
\begin{align*}
(\tAss^n_d)^! &\cong \pAss^n_{-d+n-2},
\\
(\pAss^n_d)^! &\cong \tAss^n_{-d+n-2}.
\end{align*}
\end{proposition}

Observe the shift of the degree of the generating operation.
Since $\calP$ is Koszul if and only if $\calP^!$ is, one may
reformulate the conjecture as

\begin{conjectureAp}
The operad $\pAss^n_d$ is Koszul if and only if $n \equiv d$ mod $2$.
\end{conjectureAp}

Recall that the {\em generating\/} or {\em Poincar\'e  series\/} of an
operad $\calP = \{\calP(a)\}_{a \geq 1}$ in the category of graded
vector spaces of finite type is defined by
\[
g_\calP(t) := \sum_{a \geq 1} \frac 1{a!} \chi(\calP(a)) t^a,
\]
where $\chi(-)$ denotes the Euler characteristic.

\begin{example}
\label{pent}
It is not difficult to verify that the generating
series for the operad $\tAss_d^n$ is
\[
g_{\sstAss_d^n}(t) := 
\cases{t + t^n + t^{2n-1} + t^{3n-2} + t^{4n-3} + \cdots}{if $d$  is even,}
{t-t^{n}+t^{2n-1}}{if $d$ is  odd.}
\]
We see that, for $d$ odd, $\tAss_d^n$ is \u{nontrivial} \u{onl}y \u{in}
\u{arities} $1$, $n$ and $2n-1$. This is best explained  by taking the
simplest case $n=2$ and analyzing
the operadic desuspension $\tildeAss := \osusp^{-1} \tAss^2_1$.

Recall that the {\em operadic desuspension\/} $\osusp^{-1} \calP$ of
an operad $\calP = \{\calP(a)\}_{a \geq 1}$ is the operad
$\osusp^{-1}\calP = \{\osusp^{-1}\calP(a)\}_{a \geq 1}$, where
$\osusp^{-1} \calP(a) := \mbox{sgn}_a \otimes \hskip -.4em
\downarrow^{a-1} \calP(a)$, the signum representation tensored with
the (ordinary) desuspension of the graded vector space $\calP(a)$
iterated $(a-1)$ times.  The structure operations of
$\osusp^{-1}\calP$ are induced by those
of $\calP$ in the obvious way.  The Poincar\'e series of the operad
$\calP$ and its suspension $\osusp^{-1} \calP$ are clearly related by
\begin{equation}
\label{Lei}
g_{{\ssosusp^{-1}}\calP}(t)=-g_\calP(-t).
\end{equation}

Algebras for the operad $\tildeAss$ turn out to be {\em
anti-associative\/} algebras with a degree $0$ multiplication
satisfying
\[
\label{Jaruska_je_pusinka}
a(bc) = - (ab)c,\ \mbox { for } a,b,c \in V.
\]

While $\tildeAss(1) = \bfk$, $\tildeAss(2) = \bfk[\Sigma_2]$ and 
$\tildeAss(3) = \bfk[\Sigma_3]$, 
the vanishing $\tildeAss(4)= 0$ follows from the `fake pentagon'
\begin{center}
{
\unitlength=1.3pt
\begin{picture}(110.00,100.00)(0.00,-5.00)
\thicklines
\put(20.00,0.00){\makebox(0.00,0.00){$-((a(bc))d)\hphantom{-}$}}
\put(80.00,0.00){\makebox(0.00,0.00){$(a((bc)d))$}}
\put(110.00,40.00){\makebox(0.00,0.00){$-(a(b(cd)))$}}
\put(110.00,60.00){\makebox(0.00,0.00){$\hphantom{-}(a(b(cd)))$}}
\put(50.00,87.00){\makebox(0.00,0.00){$-(ab)(cd)\hphantom{-}$}}
\put(0.00,50.00){\makebox(0.00,0.00){$(((ab)c)d)$}}
\put(20.00,10.00){\multiput(-1,-1)(2,1.5){2}{\line(-2,3){20.00}}}
\put(57.00,0.00){\multiput(0,-1)(0,2){2}{\line(-1,0){14.00}}}
\put(100.00,33.00){\multiput(1,-1)(2,-1.1){2}{\line(-2,-3){16.00}}}
\put(70.00,80.00){\multiput(1,-1)(1.3,1.7){2}{\line(3,-2){20.00}}}
\put(0.00,60.00){\multiput(-1,-1)(-2,1.5){2}{\line(3,2){30.00}}}
\end{picture}}
\end{center}
by which all $4$-fold products are trivial, as well as all $a$-fold
products for $a \geq 4$. In other words, $\tildeAss(a) = 0$ for $a
\geq 4$, so the generating series for $\tildeAss$ is therefore $t +
t^2 + t^3$. By~(\ref{Lei}), the generating series of $\tAss^2_1$ equals
\[
t - t^2 + t^3
\] 
as claimed.
\end{example}

We finally formulate the (generalized) Ginzburg-Kapranov test
\cite{ginzburg-kapranov:DMJ94}: 

\begin{theorem}
If a quadratic, not necessary binary, 
operad $\mathcal{P}$ is Koszul, then its Poincar\'e series and
the Poincar\'e series of its dual $\calP^!$ are tied by
the functional equation
\[
g_{\mathcal{P}}(-g_{\mathcal{P}^!}(-t))=t.
\]
In other words, $-g_{\mathcal{P}^!}(-t)$ is a formal inverse of
$g_{\mathcal{P}}(t)$. 
\end{theorem}

The following particular form of the GK-test is a simple consequence
of the above facts.

\begin{proposition}
If the operad $\tAss^n_d$ is Koszul, then all coefficients in the
formal inverse of \ $t-t^{n}+t^{2n-1}$ are non-negative.  
\end{proposition}

The following theorem proved in \cite{markl-remm} 
follows from the theory of analytic functions.

\begin{theorem}
\label{sec:ginzb-kapr-crit-1}
Suppose $g(z)$ is an analytic function in ${\mathbb C}$ such that
\hbox{$g(0)=0$} and $g'(0) = 1$. If the equation 
\begin{equation*}
g'(z) = 0
\end{equation*}
has \u{no} \u{real} solutions, then the formal inverse $g^{-1}(z)$ has at least
one negative coefficient.
\end{theorem}

For the generating function $g(z) :=  z-z^{n}+z^{2n-1}$ of
$\tAss^n_d$, the equation $g'(z) = 0$ reads
\[
g'(z) = 1 - n z^{n-1} + (2n-1)z^{2n-2} = 0
\]
which, after the substitution $w := z^{n-1}$, leads to
\begin{equation}
\label{Jarca1}
1 - nw + (2n-1)w^2 = 0.
\end{equation}

\begin{fact}
The discriminant $n^2 - 8n + 4$ of~(\ref{Jarca1}) is negative 
for $n \leq 7$ and positive for $n \geq 8$.
\end{fact}

The Fact explains the \u{distin}g\u{uished} \u{r\^ole} of $n=7$ resp.~$8$.  
By Theorem~\ref{sec:ginzb-kapr-crit-1}, 
the inverse of $t-t^{n}+t^{2n-1}$ has, for $n \leq 7$, 
a negative coefficient so $\tAss^n_d$ is for $d$ odd and $n \leq 7$
\u{not} \u{Koszul}.

Equation~(\ref{Jarca1}) has, for $n=8$, two real solutions,
${\mathfrak z}_1 = \sqrt[\raisebox{3pt}{\hbox{\normalsize 7\ }}]{1/3}$
and ${\mathfrak z}_2 = \sqrt[\raisebox{3pt}{\hbox{\normalsize 7\
}}]{1/5}$.  Therefore, for $n=8$ as well as for all higher $n$'s,
Theorem~\ref{sec:ginzb-kapr-crit-1} does not apply and we are unable
to prove the existence of negative coefficients in the inverse of
$z-z^{n}+z^{2n-1}$. On the contrary, the calculations given in
Section~\ref{calc} indicate that all coefficients of the inverse are
p\u{ositive}.

\section{Calculations, gaps and another conjecture}
\label{calc}

We computed, using {\tt Mathematica}, the initial parts of the
formal inverse of $t-t^{n}+t^{2n-1}$ for $n \leq 8$. We got:
\[
t+t^2+t^3-4t^5-14t^6-30t^7-33t^8+55t^9+ \cdots
\]
for $n=2$,

\[
t+t^3+2t^5+4t^7+ 5t^9-13t^{11}-147t^{13}+\cdots
\]
for $n=3$, and
\[
t+t^4+3t^7+11t^{10}+ 42t^{13}+153t^{16}+469t^{19}+690t^{22}-5967t^{25}+\cdots
\]
for $n=4$.

The first negative coefficient in the inverse of  $t-t^{n}+t^{2n-1}$ was at 
$t^{57}$ for $n=5$, at $t^{161}$ for $n=6$, and at $t^{1171}$ for
$n=7$.  For $n=8$ we \u{did} \u{not} \u{find} any negative term of
degree less than $10~000$.

To appreciate the growth of the first negative coefficient, we
introduce $\p[p] := p(n-1) +1$, $p \geq 0$, the \u{number} \u{of}
\u{instances} of the \u{multi}p\u{lication} $\mu$. The following table shows
$n$ and the corresponding $p$ such that the first negative coefficient 
occurs at $t^{\p[p]}$:
\[
\begin{array}{c|c|c|c|c|c|c|c}
n= & 2 & 3 & 4 & 5 & 6 & 7 & 8
\\
\hline
\rule{0pt}{18pt}
p= & 4 & 5 & 8 & 14 & 32 & 195 & \infty?
\end{array}
\]

The dependence of $p$ on $n$ is plotted in the following table that
clearly indicates that $p=\infty$ for $n \geq 8$, i.e.~that there are \u{no}
\u{ne}g\u{ative} \u{coefficients} in the inverse of $t - t^n + t^{2n-1}$:

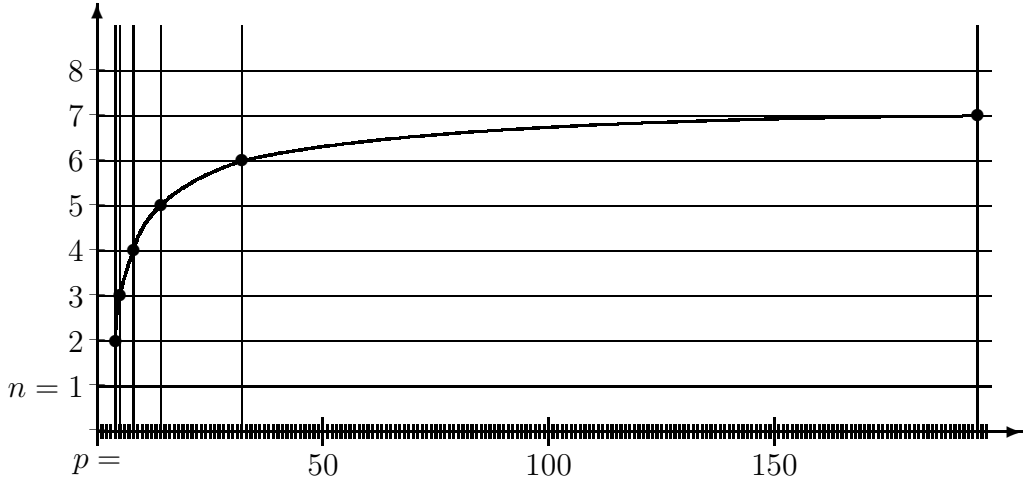
\begin{center}
\unitlength=.6mm
\begin{picture}(200,100)(-10,0)
\thicklines
\multiput(0,0)(0,10){9}{\makebox(0,0){--}}
\multiput(0,0)(1,0){198}{\makebox(0,0){\line(0,1){3}}}
\multiput(0,0)(50,0){4}{\makebox(0,0){\line(0,1){6}}}
\put(0,0){\vector(1,0){205}}
\put(0,0){\vector(0,1){95}}
\qbezier(4,20)(4,30)(8,40)
\qbezier(8,40)(11,53)(32,60)
\qbezier(32,60)(70,69)(195,70)
\thinlines
\multiput(0,0)(0,10){9}{{\line(1,0){198}}}
\put(4,0){\line(0,1){90}}
\put(5,0){\line(0,1){90}}
\put(8,0){\line(0,1){90}}
\put(14,0){\line(0,1){90}}
\put(32,0){\line(0,1){90}}
\put(195,0){\line(0,1){90}}
\put(4,20){\makebox(0,0){$\bullet$}}
\put(5,30){\makebox(0,0){$\bullet$}}
\put(8,40){\makebox(0,0){$\bullet$}}
\put(14,50){\makebox(0,0){$\bullet$}}
\put(32,60){\makebox(0,0){$\bullet$}}
\put(195,70){\makebox(0,0){$\bullet$}}
\put(-3,10){\makebox(0,0)[r]{\normalsize $n=1$}}
\put(-3,20){\makebox(0,0)[r]{\normalsize $2$}}
\put(-3,30){\makebox(0,0)[r]{\normalsize $3$}}
\put(-3,40){\makebox(0,0)[r]{\normalsize $4$}}
\put(-3,50){\makebox(0,0)[r]{\normalsize $5$}}
\put(-3,60){\makebox(0,0)[r]{\normalsize $6$}}
\put(-3,70){\makebox(0,0)[r]{\normalsize $7$}}
\put(-3,80){\makebox(0,0)[r]{\normalsize $8$}}
\put(0,-5){\makebox(0,0)[t]{\normalsize $p=$}}
\put(50,-5){\makebox(0,0)[t]{\normalsize $50$}}
\put(100,-5){\makebox(0,0)[t]{\normalsize $100$}}
\put(150,-5){\makebox(0,0)[t]{\normalsize $150$}}
\end{picture}
\end{center}

\vskip 1em

Although the GK-test does not apply for $n\geq 8$, there are some
other indications that the operad $\tAss^n_d$, $d$ even, may not be
Koszul.

\begin{example}
In~\cite{markl-remm} we \u{ex}p\u{licitl}y
\u{established} the initial part of the minimal model of $\pAss^2_1 =
(\tAss^2_{-1})^!$,
\begin{equation}
\label{2}
(\pAss^2_1,0) \leftarrow (\Gamma(E_2,E_3,\miska,E_5,\ldots),\pa).
\end{equation}
Here $E_2$ is an one-dimensional space placed in arity $2$, $E_3$ is
one-dimensional placed in arity $3$, and $E_5$ is 4-dimensional
in arity $5$. 

It was the \u{first} \u{non}-\u{trivial} \u{calculation} of the
minimal model of a non-Koszul operad. As shown in~\cite{markl-remm},
the restriction $\pa|_{E_5}$ is not quadratic but
\u{ternar}y. It then follows from the construction
of~\cite{markl:JHRS10} that the $L_\infty$-deformation complex for
$\pAss^2_1$-algebras has a non-trivial $l_3$-term.

The {\em gap} \miska\ in arity $4$ generators is caused by
the `wrong' signs in the pentagon, see Example~\ref{pent}. The
fact that it is followed by a nontrivial space $E_5$
shows that $\pAss^2_1$ is not Koszul, as follows from a proposition
below which we formulate for $n$-ary case, for arbitrary $n \geq 2$.
\end{example}

Recall  $\p[p] := p(n-1) +1$, $p \geq 0$. If $\calP$ is $n$-ary, then
$\calP(n) \not= 0$ only for $n = \p[p]$ for some $p \geq 0$, 
and for the generators $E$ of the minimal model $(\calP,0) \leftarrow
(\Gamma(E),\pa)$ clearly the same holds:
\[
E(n) \not = 0 \mbox 
{ only for $n$ of the form $n = \p[p]$ for some $p \geq 0$}.
\]

\begin{definition}
The minimal model of an $n$-ary operad has a {\em gap of 
\u{len}g\u{th} $d \geq 1$\/} if
there is $q \geq 2$ such that 
\[
E(\p[p]) = 0 \mbox { for } q \leq p \leq q+d-1
\]
while 
\[
E(\p[q-1]) \not= 0 \not = E(\p[q+d]).
\] 
\end{definition}

The model of $\pAss^2_1$ is of the form
$(\Gamma(E_{\p[1]},E_{\p[2]},\miska,E_{\p[4]},\ldots),\pa)$ with
non-trivial $E_{\p[2]}$ and $E_{\p[4]}$. So it has a gap of length $1$
-- with $d = 1$, $q=3$ in the above definition.

\begin{proposition}
Suppose that the minimal model of a quadratic $n$-ary operad 
$\calP$ has a gap of finite length. Then $\calP$ 
is \u{not} \u{Koszul}.
\end{proposition}

\begin{proof}
Suppose that $\calP$ is Koszul and let $(\calP,0) \leftarrow
(\Gamma(E),\pa)$ be its minimal model. 
It follows from Definition~\ref{Zitra_budu_s_Jaruskou!} and the
uniqueness of the minimal model for operads 
\cite[Theorem~II.3.126]{markl-shnider-stasheff:book} 
that the collection $E$ is the (suitably
suspended) Koszul dual $\calP^!$. 
The operad $\calP^!$ is $n$-ary, too, so $\calP^!(\p[q]) = 0$
for some $q \geq 2$ implies $\calP^!(\p[p]) = 0$ \u{for} \u{all} $p
\geq q$.  Thus $\calP^!$ and therefore also $E$ cannot have a gap of a
{\em finite} length.
\end{proof}

The strategy we suggest is to study the gaps in the minimal model of
$\pAss^n_d$ with $n \not \equiv d$ mod $2$. Their existence would
imply non-Koszulity of $\pAss^n_d$, as well as the non-Koszulity of their
Koszul duals $\tAss^n_d$, $d$ odd, thus establishing Conjecture~A.
It is not difficult to prove the following:

\begin{proposition}
\label{Jaruska_mi_udelala_svickovou!!}
Let $\calP$ be an arbitrary, \u{not} \u{necessaril}y \u{Koszul}, 
operad with $\calP(1) =
\bfk$, and $(\calP,0) \leftarrow (\Gamma(E),\pa)$ its
minimal model. The Poincar\'e series $g_\calP(t)$ of $\calP$
is related with the generating function
\[
g_E(t) := -t + \sum_{a \geq 2} \frac 1{a!} \chi(E(a))t^a
\] 
of the $\Sigma$-module $\{E(a)\}_{a \geq 2}$ by the functional
equation
\[
g_\calP(-g_E(t)) = t.
\]
\end{proposition}

The above theorem enables one to calculate the Poincar\'e  series of
the collection of generators of
the minimal model of $\calP$ from the generating series of $\calP$. It
clearly implies the GK-criterion.

\begin{example}
It happens that $\pAss^2_1 = \tAss^2_1$, so the
generating series of 
$\pAss^2_1$ is 
\[
g_{\sspAss_1^2}(t) =  t-t^2+t^3.
\]
One can compute the formal inverse of this function as 
\[
t+t^2+t^3-4t^5-14t^6-30t^7-33t^8+55t^9+ \cdots.
\]
The absence of the $t^4$-term together with the presence of the
$t^5$-term ``shows'' the gap of length $1$ in the
minimal model of $\pAss^2_1$.
\end{example}

We do not know any closed formula for the generating series of
$\pAss^n_d$ with $n \not \equiv d$ mod $2$, \u{$n > 2$}. We however
wrote a script for {\tt Mathematica} that calculates
it, but its applicability is drastically limited by
computers available. We established the generating series for $\pAss^3_0$~as
\[
t+t^3+2t^5+4t^7+5t^9+6t^{11}+7t^{13}+8t^{15}+\cdots,
\]
the generating series of $\pAss^4_1$ as
\[
t-t^4+3t^{7}-11t^{10}+42t^{13}-153t^{16}+565t^{19}+\cdots,
\] 
the generating series of $\pAss^5_0$ as
\[
t+t^5+4t^9+21t^{13}+123t^{17}+759t^{21}+\cdots,
\] 
the generating series of $\pAss^7_0$ as
\[
t+t^7+6t^{13}+50t^{19}+481t^{25}+\cdots
\]
and the generating series of $\pAss^9_0$ as
\[
t+t^9+8t^{17}+91t^{25}+1207t^{33}+\cdots.
\]
By calculating the formal inverses of the above series, we get the
following Poincar\'e series of the generators for the minimal models: 
\[
t+t^2+t^3 +\boxed{ 0\t4}  -4t^5-14t^6-30t^7-33t^8+55t^9+ \cdots
\]
for $\pAss^2_1$ (we already know this),
\[
t - \t3 + \t5 + \boxed{0\t7 + 0\t9} -19\t{11} + 112 \t{13} - 336 \t{15} + \cdots
\]
for $\pAss^3_0$, 
\[
t+\t4+\t7 +\boxed{0\t{10}+ 0\t{13} + 0\t{16}} - 96\t{19} + \cdots
\]
for $\pAss^4_1$, 
\[
t-\t5+\t9 +\boxed{0\t{13} + 0 \t{17} + 0\t{21} +\ ?} +{\rm O}[t^{25}]
\] 
for $\pAss^5_0$.  The vanishing of the boxed terms imply, by
Proposition~\ref{Jaruska_mi_udelala_svickovou!!}, 
that the Euler characteristics of the corresponding pieces of
the generating collection is zero. It indicates that the
minimal models are of the form
\begin{eqnarray*}
(\pAss^2_1,0)& \leftarrow&
  (\Gamma(E_{\p[1]},E_{\p[2]},\miska,E_{\p[4]},\ldots),\pa),
\\
(\pAss^3_0,0)& \leftarrow& 
(\Gamma(E_{\p[1]},E_{\p[2]},\miska,\miska,E_{\p[5]},\ldots),\pa),
\\
(\pAss^4_1,0)& \leftarrow& 
(\Gamma(E_{\p[1]},E_{\p[2]},\miska,\miska,\miska,E_{\p[6]},\ldots),\pa),
\\
(\pAss^5_0,0)& \leftarrow& 
(\Gamma(E_{\p[1]},E_{\p[2]},\miska,\miska,\miska,\miskaot,
\raisebox{.2em}{\boxed{?}},\ldots),\pa).
\end{eqnarray*}
Our computation did not go beyond $n\geq 6$, due to limitations of
computer memory. The results for small $n$'s however suggest
that the gap grows {\em \u{linearl}y\/} with $n$, leading to

\begin{conjectureB}
The minimal model of $\pAss^n_d$, $n \not \equiv d$ mod $2$, has a gap
of length $n-1$.  
\end{conjectureB}

The conjecture would obviously imply the non-Koszulity of $\tAss^n_d$, for
$d$~even. If it is so, then $\tAss^8_1$ will be the first
\u{exam}p\u{l}\u{e} 
of a \u{non}-\u{Koszul} operad whose non-Koszulity \u{was} \u{not} 
established by the Ginzburg-Kapranov criterion.


\def\cprime{$'$}

\end{document}